\documentclass[a4paper,12pt]{article}

\usepackage[utf8]{inputenc}
\usepackage[english]{babel}
\usepackage{lmodern}
\usepackage[T1]{fontenc}

\usepackage{graphicx, hyperref}
\usepackage{cite}
\usepackage{amsthm, amsmath, amssymb, amscd, eucal, stmaryrd}
\input xy
\xyoption{all}

\newtheorem{teo}{Theorem}[section]
\newtheorem{defin}[teo]{Definition}
\newtheorem{prop}[teo]{Proposition}

\theoremstyle{definition}

\newtheorem{oss}[teo]{Remark}
\numberwithin{equation}{section}

\title{ \textbf {Symplectic rigidity and weak commutativity}}

\author{ \sc Simone Vazzoler \qquad Franco Cardin \\ { \sc 
}
\\ Dipartimento di Matematica Pura ed Applicata\\ Via Trieste, 63 - 35121 Padova, Italy }

\begin{document}

\maketitle


\begin{abstract}
An alternative proof of Eliashberg-Gromov's $C^0$-rigidity theorem is presented and a new notion of weak Lie brackets for Hamiltonian vector fields is proposed and compared.
\end{abstract}

A notion of Hamiltonian $C^0$-commutativity, together with a related theorem linking it to the standard Poisson brackets, has been introduced in \cite{CaVi08} in connection with the problem of finding variational solutions of multi-time Hamilton-Jacobi equations. The main theorem on this new environment is based essentially on the use of Viterbo's capacities. In that paper it has been foreseen that from that $C^0$-commutativity framework Eliashberg-Gromov's theorem on symplectic rigidity could follow. An interesting proof of this fact has been recently worked out by V. Humili\`ere (see \cite{Hu09}) using the concept of pseudo-representations. 
 \\
  The question on the $C^0$-closure of the group of symplectomorphisms is widely considered as a starting point for the study of symplectic topology, thus is important to enrich this particular area with new proofs, eventually trying to give further elucidations to the subject. In this note, moving from the above $C^0$-commutativity, an alternative proof of Eliashberg-Gromov's theorem is presented which is based on simple algebraic arguments. 
  \\
  A general notion of weak commutativity of vector fields (not necessarily Hamiltonian) is presented in the work by Rampazzo and Sussmann (see \cite{RaSu01}), where they extend the usual Lie brackets even in the case of Lipschitz vector fields. In \cite{CaVi08} has been suggested the existence of a possible relation between $C^0$-commutativity and the notion presented in \cite{RaSu01}. Here, we enter adding some details on this matter. Firstly, on the line of thought of  \cite{RaSu01}, we introduce the notion of weak Hamiltonian vector field and afterwards the notion of weak Lie brackets; this construction makes our definition coinciding with the one given in \cite{RaSu01}. Lastly, we see that this setting is right suitable to provide the equivalence between $C^0$-commutativity and weak commutativity of vector fields for Hamiltonians of class $C^{1,1}$.

\section{Generating Functions Quadratic at Infinity (GFQI)}

Let $M$ be a paracompact $n$ dimensional manifold and $S\in C^2(M;\mathbb{R})$. In what follows we suppose that the Palais-Smale condition holds:
\begin{itemize}
\item[(PS)]  The pair $(M,S)$ satisfies the (PS) condition if every sequence $\{x_k\}_{k\in\mathbb{N}}$ such that $\displaystyle\lim_{n\rightarrow\infty}\|D S(x_k)\|=0$ and $S(x_k)$ is bounded, admits a convergent subsequence.
\end{itemize}
We notice that paracompactness allows us to choose a Riemannian metric $g$ on $M$ such that, for example,  $\|D S(x)\|=\sqrt{\langle g(x)D S(x),D S(x)\rangle}$.
We will denote $S^\lambda$ the sublevel set relative to $\lambda$, i.e. $S^\lambda=\{q\in M\,|\,S(q)\leq \lambda\}$, and for every $\alpha\in H^*(S^b,S^a)\setminus \{0\}$ we define 
\[
c(\alpha,S)=\inf\{\lambda\in [a,b]\,:\,\,i^*_\lambda\alpha\neq 0\}
\]
where $i^*_\lambda:H^*(S^b,S^a)\rightarrow H^*(S^\lambda,S^a)$ is the map induced by the natural inclusion $i_\lambda:S^\lambda\hookrightarrow S^b$ (if $M$ is not compact we will work with compactly supported differential forms).
We recall some important properties of $c(\alpha,S)$:
\begin{itemize}
\item[(i)] $c(\alpha,S)$ is a critical value of $S$: it is commonly said the ``min-max'' critical value;
\item[(ii)] $c(u\cdot v,S_1+S_2)\geq c(u,S_1)+c(v,S_2)$;
\item[(iii)] if $\alpha\in H_{n-q}(M)$ is the Poincar\'e dual of $\mu\in H^q(M)$, then $c(\mu,-S)=-c(\alpha,S)$; in particular:
\item[(iv)] if $1\in H^0(M)$ and $\mu\in H^n(M)$, then $c(1,-S)=-c(\mu,S)$.
\end{itemize}
For the proofs see \cite{Vi92} or \cite{Vi06}. Now we consider a closed differentiable manifold $M$ and denote with $L$ a Lagrangian submanifold of $T^*M$ isotopic to the zero section $O_M$ by means of a compactly supported time one Hamiltonian flow $\varphi^1$.
\begin{defin}
A smooth function $S:M\times\mathbb{R}^k\rightarrow \mathbb{R}$ is a generating function quadratic at infinity (GFQI) for a Lagrangian submanifold $L$ if:
\begin{itemize}
\item[(i)] the map
\[
(q;\xi)\rightarrow \frac{\partial S}{\partial \xi}(q;\xi)
\]
has zero as a regular value;
\item[(ii)] the map
\begin{align*}
i_s:\Sigma_s\subset M\times \mathbb{R}^k&\rightarrow T^*M\\
(q;\xi)&\mapsto(q;\frac{\partial S}{\partial q}(q;\xi))
\end{align*}
has image $i_s(\Sigma_s)=L$, where $\Sigma_s=\{(q;\xi):\frac{\partial S}{\partial\xi}(q;\xi)=0\}$.
\item[(iii)] for $|\xi|>C$ 
\[
S(q;\xi)=\xi^T Q\xi
\]
where $\xi^T Q\xi$ it is a non degenerate quadratic form;
\end{itemize}  
\end{defin}
It is well known (see again \cite{Vi92}) that a generating function is unique up to three fundamental operations: diffeomorphisms on the fibers and addition of quadratic forms and constants.
In more details, let $S_1,S_2$ be two GFQI. We easy see that $S_1$ and $S_2$ are \textit{equivalent}, i.e. they draw the same $L$, if there exists a diffeomorphism
\begin{align*}
\Phi:M\times\mathbb{R}^k&\rightarrow M\times\mathbb{R}^k\\
(q;\xi)&\mapsto(q;\phi(q;\xi))
\end{align*}
such that $S_1(q;\phi(q;\xi))=S_2+c$ with $c\in\mathbb{R}$. Moreover, we see also that a GFQI $S_2$  is still equivalent to $S_1$  if
\[
S_2(q;\xi,\eta)=S_1(q;\xi)+\eta^T B\eta
\]
where $\eta^T B\eta$ it is a non degenerate quadratic form on the fibers; $S_2$ is said \textit{stabilization} of $S_1$. The following theorems ensure the existence and unicity of a GFQI for $L=\varphi^1(O_M)$, precisely up to the three operations above.
\begin{teo}
Let $O_{M}$ be the zero section of $T^*M$ and $(\varphi^t)$ a Hamiltonian flow. Then the Lagrangian submanifold $\varphi^1(O_{M})$ admits a GFQI.
\end{teo}
\begin{proof}
\cite{Si86}
\end{proof}
\begin{teo}
Let $S_1$ and $S_2$ be two GFQI for $L=\varphi^1(O_M)$. Then, up to the above three operations, $S_1$ and $S_2$ are equivalent.
\end{teo}
\begin{proof}
\cite[page 688]{Vi92}
\end{proof}
\begin{oss}
In literature sometimes local generating functions $S(q,\xi)$ satisfying $(i)$ and $(ii)$, but not $(iii)$, are called Morse families.
It was known, before \cite{Vi92}, that the above three operations are locally characterizing any Morse family, see \cite {We79} and \cite{LiMa87}.
\end{oss}
If $S$ is a GFQI and $c\in\mathbb{R}$ is large enough, one has
\[
H^*(S^c,S^{-c})\simeq H^*(M)\otimes H^*(D^-,\partial D^-)
\]
where $D^-$ is the unitary disc in the negative eigenspace of $Q$. So, after choosing $\alpha\in H^*(M)\setminus\{0\}$, we can associate to it $\alpha\otimes T\in H^*(S^c,S^{-c})$, where $T$ is a generator of $H^*(D^-,\partial D^{-})\simeq\mathbb{R}$.

\section{The $\gamma$ and $\widehat{\gamma}$ metrics and $C^0$-commuting Hamiltonians}

All the proofs in this section can be found in \cite{CaVi08}.
\begin{defin}
Let $S$ be a GFQI for $L=\varphi^1(O_M)$. We define 
\begin{align*}
\gamma(L)=c(\mu,S)-c(1,S),\, 1\in H^0(M),\,\mu\in H^n(M)
\end{align*}
\end{defin}
It is important to remark that this definition is well posed: even if $c(\mu,S)$ and $c(1,S)$ depends on $S$, the difference does not depend on the function. On the set of the Lagrangian submanifolds of $T^*M$ isotopic to the zero section, named $\mathcal{L}$, we can define a metric.
\begin{defin}
Given $L_1,L_2\in\mathcal{L}$ we define
\[
\gamma(L_1,L_2)=c(\mu,S_1\ominus S_2)-c(1,S_1\ominus S_2)
\] 
where $(S_1\ominus S_2)(q;\xi_1,\xi_2)=S_1(q;\xi_1)-S_2(q;\xi_2)$.
\end{defin}
In \cite{Vi92} it is shown that $\gamma$ is a metric on the set $\mathcal{L}$. If we define $\mathcal{H}_c(T^*M)=C^{\infty}_c([0,1]\times T^*M,\mathbb{R})$ (i.e. the set of the time dependent Hamiltonians with compact support) and set $\mathcal{HD}_c(T^*M)$ to be the group of the time one maps of $\mathcal{H}_c(T^*M)$, we can extend the $\gamma$ metric to $\mathcal{HD}_c(T^*M)$ in this way
\begin{gather*}
\widehat{\gamma}(\varphi)=\sup\{\gamma(\varphi(L),L): L\in\mathcal{L}\}\\
\widehat{\gamma}(\varphi,\psi)=\widehat{\gamma}(\varphi\psi^{-1})
\end{gather*}
\begin{prop}
$\widehat{\gamma}$ defines a bi-invariant metric on $\mathcal{HD}_c(T^*M)$.
\end{prop}
We will say that $\varphi_n$ c-converges to $\varphi$ (and we will write $\varphi_n\overset{c}{\rightarrow}\varphi$) if
\[
\displaystyle\lim_{n\rightarrow\infty}\widehat{\gamma}(\varphi_n,\varphi)=0
\]
We can extend $\widehat{\gamma}$ to $\mathcal{H}_c(T^*M)$: if we choose $H$ with flow $\psi^t$, we can define $\widehat{\gamma}(H)=\sup_{t\in [0,1]}\widehat{\gamma}(\psi^t)$. The following inequality (see \cite{CaVi08})
\begin{equation}\label{inequality}
\widehat{\gamma}(\varphi)\leq \|H\|_{C^0}
\end{equation}
holds, where $\varphi$ is the flow at time one of the Hamiltonian $H(t,q,p)$ and 
\begin{equation*}
\|H\|_{C^0}=\sup_{(t,q,p)}H(t,q,p)-\inf_{(t,q,p)}H(t,q,p)
\end{equation*}
If $\varphi_n$ and $\varphi$ are the time one flows of $H_n$ and $H$ respectively and we have that $H_n\rightarrow H$ in the $C^0$ topology, then $\varphi_n\overset{c}{\rightarrow}\varphi$.
\section{$C^0$-commuting Hamiltonians and Eliashberg-Gromov's theorem}

\begin{defin}{\label{C^0DEF}}
Let $H,K$ be two autonomous Hamiltonians. We will say that $H$ and $K$ $C^0$-commutes if there exist two sequences $H_n,K_n$ of $C^1$ Hamiltonians $C^0$-converging to $H$ and $K$ respectively 
such that, in the $C^0$-topology:
\[
\displaystyle\lim_{n\rightarrow\infty}\{H_n,K_n\}=0
\]
\end{defin}
\begin{prop}
Let $H,K$ be two autonomous Hamiltonians of class $C^{1,1}$ and suppose that $\{H,K\}$ is small in the $C^0$ norm. If $\varphi^t,\psi^s$ are the flows of $H$ and $K$ respectively, then the isotopy $t\mapsto \varphi^t\psi^s\varphi^{-t}\psi^{-s}$ is generated by a $C^0$ small Hamiltonian.
\end{prop}
Definition \ref{C^0DEF} is a good extension of the standard Poisson brackets commutation since the  following theorem does hold.
\begin{teo}[Cardin, Viterbo \cite{CaVi08}]\label{teoCaVi08} Let $H$ and $K$ be two Hamiltonians of class $C^{1,1}$. If they $C^0$-commute then $\{H,K\}=0$ in the usual sense.
\end{teo}
The last theorem can be extended also to the affine at infinity case (see \cite{Hu09} Lemma 10). Another important generalization of the previous theorem can be
found in \cite{EnPo10}.
In what follows we will consider only sequences of symplectomorphisms $n\mapsto \Phi^{(n)}$ that are bounded deformations of the identity, more precisely such that $\mbox{supp}(\Phi^{(n)})-\mbox{Id}$ is compact. 
\begin{teo}[Symplectic rigidity, \cite{El81},    \cite{El87},   \cite{EkHo89},    \cite{Gr85}]
The group of compactly supported symplectomorphisms is $C^0$-closed in the group of all diffeomorphisms of $\mathbb{R}^{2d}$.
\end{teo}
\begin{proof}
To fix the notations: $(q,p)=(q_1,\ldots,q_d,p_1,\ldots,p_d)\in\mathbb{R}^{2d}$ and denote with
\[
(Q_1^{(n)}(q,p),\ldots,Q_d^{(n)}(q,p),P_1^{(n)}(q,p),\ldots,P_d^{(n)}(q,p))
\]
a sequence of symplectic transformations $C^0$-converging to $(Q(q,p),P(q,p))$.\\
Note that we have to prove only $\{Q_i,P_i\}=1$. In fact, the other relations $\{ Q_i,Q_j  \}=0=\{  P_i,P_j \}$, and $\{Q_i,P_j\}=0$ for $i\neq j$, are automatically satisfied using Theorem \ref{teoCaVi08} and Lemma 10 in \cite{Hu09}. Now we define a new sequence (using the previous one)
\[
\begin{cases}
\displaystyle\widetilde{Q}_i^{(n)}=Q_i^{(n)}+\frac{1}{\sqrt{d}}\sum_{k=1}^dP_k^{(n)}\\
\displaystyle\widetilde{P}_i^{(n)}=P_i^{(n)}+\frac{1}{\sqrt{d}}\sum_{k=1}^dQ_k^{(n)}
\end{cases}
\]
Clearly $\{\widetilde{Q}_i^{(n)},\widetilde{P}_i^{(n)}\}=0$, in fact
\begin{gather*}
\{\widetilde{Q}_i^{(n)},\widetilde{P}_i^{(n)}\}=\{Q_i^{(n)},P_i^{(n)}\}+\frac{1}{\sqrt{d}}\sum_{k=1}^d(\{P_k^{(n)},P_i^{(n)}\}+\{Q_i^{(n)},Q_k^{(n)}\})\\
+\frac{1}{d}\sum_{k=1}^d\{P_k^{(n)},Q_k^{(n)}\}=1-1=0
\end{gather*}
Using again Theorem \ref{teoCaVi08} and Lemma 10 in \cite{Hu09}, we get $\{\widetilde{Q}_i,\widetilde{P}_i\}=0$: passing to the limit,
\begin{gather*}
\begin{array}{rcl}
\{\widetilde{Q}_i,\widetilde{P}_i\}&=&\{Q_i,P_i\}+\frac{1}{\sqrt{d}}\sum_{k=1}^d(\{P_k,P_i\}+\{Q_i,Q_k\})
+\frac{1}{d}\sum_{k=1}^d\{P_k,Q_k\},  \\
\\
{}&{=}&  \{Q_i,P_i\}+\frac{1}{d}\sum_{k=1}^d\{P_k,Q_k\}=0
\end{array}
\end{gather*}
Define (just to semplify the notations) $C_i(q,p)=\{Q_i,P_i\}$. For every fixed $(q,p)\in\mathbb{R}^{2d}$ the last homogeneous linear system reads
\[
\begin{pmatrix}
d-1 & -1 & \ldots & -1\\
-1 & d-1 & \ldots & -1\\
\vdots & \vdots & \ddots & \vdots\\
-1 & -1 & \ldots  & d-1 
\end{pmatrix}
\begin{pmatrix}
C_1\\
C_2\\
\vdots \\
C_d
\end{pmatrix}=
\begin{pmatrix}
0\\
0\\
\vdots \\
0
\end{pmatrix}
\]
that has $C_1=C_2=\ldots=C_d$ as solution; in fact, the $d\times d$ matrix has determinant equal to zero: if we sum the last $d-1$ rows we get the opposite of the first row; in particular the rank of the matrix is $d-1$, so the subspace of solutions has dimension $1$, spanned by the above equal components vector. 

Recalling the Jacobi identity
\[
\{f,\{g,h\}\}+\{g,\{h,f\}\}+\{h,\{f,g\}\}=0
\]
we obtain, considering terms like $\{Q_i,\{Q_j,P_j\}\}$ for $i\ne j$,
\[
0=\{Q_i,\{Q_j,P_j\}\}+\{Q_j,\{P_j,Q_i\}\}+\{P_j,\{Q_i,Q_j\}\}
\]
and since $\{Q_i,P_j\}=\{Q_i,Q_j\}=0$, we get
\[
\{Q_i,\{Q_j,P_j\}\}=0
\]
Analogously, starting with $\{P_i,\{Q_j,P_j\}\}$ we obtain
\[
\{P_i,\{Q_j,P_j\}\}=0
\]
Once we have posed $C_1(q,p)=C_2(q,p)=\ldots=C_d(q,p)=C(q,p)$, using the previous relations, we have 
\[
\begin{cases}
\{Q_1,C\}=0\\
\{Q_2,C\}=0\\
\mbox{\phantom{PP}}\vdots \\
\{P_{d-1},C\}=0\\
\{P_d,C\}=0
\end{cases}
\]
that is a homogeneous linear system of the type $A\cdot DC=0$
\[
\begin{pmatrix}
-Q_{1,p_1} & \ldots & -Q_{1,p_d} & Q_{1,q_1} & \ldots & Q_{1,q_d}\\
-Q_{2,p_1} & \ldots & -Q_{2,p_d} & Q_{2,q_1} & \ldots & Q_{2,q_d}\\
\vdots & \ddots & \vdots & \vdots & \ddots & \vdots \\
-P_{d,p_1} & \ldots & -P_{d,p_d} & P_{d,q_1} & \ldots & P_{d,q_d} 
\end{pmatrix}
\begin{pmatrix}
C_{,q_1}\\C_{,q_2}\\ \vdots \\ C_{,p_d}
\end{pmatrix}=
\begin{pmatrix}
0\\ 0\\ \vdots \\ 0
\end{pmatrix}
\]
From the fact that $\Phi:(q,p)\mapsto (Q_1,\ldots,Q_d,P_1,\ldots,P_d)$ is a diffeomorphism we have $\det A\neq 0$ (because $A=D\Phi\cdot\mathbb{E}$ where $\mathbb{E}$ is the symplectic matrix) and so $C(q,p)=C$, a constant. It remains to show that $C=1$. This comes from the fact that the limit is a deformation of the identity, i.e. outside a compact set of $\mathbb{R}^{2d}$ we have $(Q_1,\ldots,Q_d,P_1,\ldots,P_d)=(q_1,\ldots,q_d,p_1,\ldots,p_d)$ and so $\{q_i,q_j\}=\{p_i,p_j\}=0$ and $\{q_i,p_j\}=\delta_{ij}$ outside this compact set. From the fact that the Poisson brackets are (at least) continuous then we must have $C=1$.
\end{proof}
\section{Connection with Lie brackets}

It is possible to extend the connection between the Poisson brackets and the Lie brackets. We proceed in this way: we define a ``weak'' Hamiltonian vector field (so that it is defined even if the Hamiltonian is not $C^1$ but only Lipschitz) and then we extend the definition of Lie brackets using this vector field. First of all we recall here the definition of weak Lie brackets by Rampazzo and Sussmann (see \cite{RaSu01} section 5) and then, following that line of thought, we propose a definition of weak Hamiltonian vector field.
\begin{defin}[Rampazzo, Sussmann \cite{RaSu01}]\label{RS}
Let $f,g$ be two locally Lipschitz vector fields on $\mathbb{R}^n$. We define the Lie bracket of $f$ and $g$ at $x$, and we will write $[f,g](x)$, to be the convex hull of the set of all vectors
\[
v=\displaystyle\lim_{j\rightarrow\infty}(Df(x_j)\cdot g(x_j)-Dg(x_j)\cdot f(x_j))
\] 
for all sequences $\{x_j\}_{j\in\mathbb{N}}$ such that
\begin{enumerate}
\item $x_j\in$ Diff$(f)\cap$ Diff$(g)$ for all $j$;
\item $\lim_{j\rightarrow\infty}x_j=x$;
\item the limit $v$ exists. 
\end{enumerate}
\end{defin}
Inspired by Definition \ref{RS}, we introduce
\begin{defin}
Let $H$ be a Lipschitz continuous compactly supported Hamiltonian on $\mathbb{R}^{2n}$. We define the weak Hamiltonian vector field $\mathbb{X}_H$ as the following set valued vector field:
\[
\mathbb{X}_H:=
\begin{cases}
\mathbb{E}D H(q,p), \mbox{ if } (q,p)\in\mbox{Diff}(H)\\
\displaystyle\mbox{\textbf{ch}}\{v: v=\lim\mathbb{E}D H(q_j,p_j), \forall (q_j,p_j)\rightarrow(q,p)\}, \mbox{ otherwise }
\end{cases}
\]
where $\mathbb{E}$ is the standard symplectic matrix, \textbf{ch} is the convex hull and the sequences $(q_j, p_j)$ belong to $\mbox{Diff}(H)$.
\end{defin}
The weak Hamiltonian vector field has two properties:
\begin{itemize}
\item[(i)]$\mathbb{X}_H(q,p)$ is a non empty, closed and convex set of $\mathbb{R}^{2n}$;
\item[(ii)]if $H\in C^1$ then the weak vector field coincides with the usual Hamiltonian vector field.
\end{itemize}
It is well known that if $H,K$ are two $C^2$ Hamiltonians then the following equality holds:
\[
[X_H,X_K]=X_{\{H,K\}}
\]
We can extend this relation to the case when $H$ and $K$ are $C^{1,1}$.
\begin{defin}
Let $H,K\in C^{1,1}$. We define the weak Lie brackets as
\[
\llbracket X_H,X_K\rrbracket :=\mathbb{X}_{\{H,K\}}
\] 
\end{defin}
Note that $X_H$ and $X_K$ are well defined, but their Lie brackets are not (because $X_H$ and $X_K$ are only Lipschitz vector fields). Because of the definition of $\mathbb{X}_H$ the following proposition holds:
\begin{prop}
The weak Lie brackets have two properties:
\begin{itemize}
\item[(i)] the definition coincides with the one given in \cite{RaSu01} (i.e. these weak Lie brackets are compatible with Rampazzo-Sussmann's ones); 
\item[(ii)] if $H,K\in C^2$ then $\llbracket X_H,X_K\rrbracket = [X_H,X_K]$.
\end{itemize}
\end{prop}
In this constructed framework we are ready to state the following result involving the $C^0$-commutativity:
\begin{prop}
If $H,K$ are $C^0$-commuting $C^{1,1}$ Hamiltonians then weak commutativity of vector fields holds: $\llbracket X_H,X_K\rrbracket =0$. Conversely, if we have $\llbracket X_H,X_K\rrbracket =0$ then $\{H,K\}=0$. In particular $H,K$  $C^0$-commute.  
\end{prop}
\begin{proof}
If $H,K$ $C^0$-commute then we have $\{H,K\}=0$ and so we get easily from the definition $\llbracket X_H,X_K\rrbracket=0$. Conversely if $\llbracket X_H,X_K\rrbracket=0$ then their flows commute (see \cite{RaSu01}) and so $H,K$ commute and in particular $C^0$-commute.  
\end{proof}

\bibliography{bibliography}{}

\begin{thebibliography}{10}

\bibitem{CaVi08}
F.~Cardin and C.~Viterbo.
\newblock Commuting {H}amiltonians and {H}amilton-{J}acobi multi-time
  equations.
\newblock {\em Duke Math. J.}, 144(2):235--284, 2008.

\bibitem{EkHo89}
I.~Ekeland and H.~Hofer.
\newblock Symplectic topology and {H}amiltonian dynamics.
\newblock {\em Mathematische Zeitschrift}, 200:355--378, 1989.
\newblock 10.1007/BF01215653.

\bibitem{El81}
Y.~M. Eliashberg.
\newblock Rigidity of symplectic and contact structures.
\newblock {\em Preprint}, 1981.

\bibitem{El87}
Y.~M. Eliashberg.
\newblock A theorem on the structure of wave fronts and its application in
  symplectic topology.
\newblock {\em Funktsional. Anal. i Prilozhen.}, 21(3):65--72, 1987.

\bibitem{EnPo10}
M.~Entov and L.~Polterovich.
\newblock {$C^0$}-rigidity of {P}oisson brackets.
\newblock In {\em Symplectic topology and measure preserving dynamical
  systems}, volume 512 of {\em Contemp. Math.}, pages 25--32. Amer. Math. Soc.,
  Providence, RI, 2010.

\bibitem{Gr85}
M.~Gromov.
\newblock Pseudoholomorphic curves in symplectic manifolds.
\newblock {\em Invent. Math.}, 82(2):307--347, 1985.

\bibitem{Hu09}
V.~Humili\`ere.
\newblock Hamiltonian pseudo-representations.
\newblock {\em Comment. Math. Helv.}, 84(3):571--585, 2009.

\bibitem{LiMa87}
P.~Libermann and C.-M. Marle.
\newblock {\em Symplectic geometry and analytical mechanics}, volume~35 of {\em
  Mathematics and its Applications}.
\newblock D. Reidel Publishing Co., Dordrecht, 1987.
\newblock Translated from the French by Bertram Eugene Schwarzbach.

\bibitem{RaSu01}
F.~Rampazzo and H.~J. Sussmann.
\newblock Set-valued differentials and a nonsmooth version of chow's theorem.
\newblock In {\em In Proc. of 40th Conf. Decision Control}, pages 2613--2618.
  IEEE publications, 2001.

\bibitem{Si86}
J.-C. Sikorav.
\newblock Sur les immersions lagrangiennes dans un fibr\'e cotangent admettant
  une phase g\'en\'eratrice globale.
\newblock {\em C. R. Acad. Sci. Paris S\'er. I Math.}, 302(3):119--122, 1986.

\bibitem{Vi92}
C.~Viterbo.
\newblock Symplectic topology as the geometry of generating functions.
\newblock {\em Mathematische Annalen}, 292(1):685--710, 1992.
\newblock 10.1007/BF01444643.

\bibitem{Vi06}
C.~Viterbo.
\newblock Symplectic topology and {H}amilton-{J}acobi equations.
\newblock In P.~Biran, O.~Cornea, and F.~Lalonde, editors, {\em Morse Theoretic
  Methods in Nonlinear Analysis and in Symplectic Topology}, volume 217, pages
  439--459. Springer Netherlands, 2006.

\bibitem{We79}
A.~Weinstein.
\newblock {\em Lectures on symplectic manifolds}, volume~29 of {\em CBMS
  Regional Conference Series in Mathematics}.
\newblock American Mathematical Society, Providence, R.I., 1979.
\newblock Corrected reprint.

\end{thebibliography}
\bibliographystyle{abbrv}

\end{document}